\documentclass[11pt,twoside]{article}

\date{10 February 2012}
\def\titlename{Approximate Amenability of Segal Algebras}

\title{\titlename}

\def\authname{Mahmood Alaghmandan}
\author{{\normalsize\sc \authname}}

\usepackage{amsmath}
\usepackage{yfonts}
\usepackage{xcolor}

\definecolor{El}{rgb}{.3,.7,1}
\usepackage[pdftex,
            pdfauthor={Mamhood},
            pdftitle={Approximate Amenability of Segal Algebras},
            pdfsubject={Glossaries Package for LaTeX},
            pdfkeywords={LaTeX, Package, Glossary},
            pdfproducer={Latex with hyperref},
            pdfcreator={PDFLaTeX},
            pdfduplex=DuplexFlipLongEdge,
            linkcolor=blue,
            citecolor=purple,
            colorlinks=true]{hyperref}
            
            \newcounter{pulse}[section]
\numberwithin{pulse}{section}
\numberwithin{equation}{section}

\newtheorem{theorem}[pulse]{\bf Theorem}
\newtheorem{proposition}[pulse]{\bf Proposition}
\newtheorem{lemma}[pulse]{\bf Lemma}

\newtheorem{cor}[pulse]{\bf Corollary}

\newtheorem{dummy-eg}[pulse]{\bf Example}
\newtheorem{dummy-rem}[pulse]{\bf Remark}
\newenvironment{eg}{\begin{dummy-eg}\upshape}{\end{dummy-eg}\ignorespacesafterend}
\newenvironment{rem}{\begin{dummy-rem}\upshape}{\end{dummy-rem}\ignorespacesafterend}

\newtheorem{dummy-def}[pulse]{\bf Definition}

\usepackage{amsmath,amssymb}

\newenvironment{dfn}{\begin{dummy-def}\upshape}{\end{dummy-def}\ignorespacesafterend}

\newenvironment{proof}{\noindent{\it Proof.}\/}{\hfill$\Box$\newline\ignorespacesafterend}

\newcommand{\supp}{\operatorname{supp}}

\renewcommand{\ker}{\operatorname{Ker}}

\newcommand{\Lp}{{{\cal L}^p(\widehat{G})}}

\newcommand{\Cp}{{{\textswab C}^p(G)}}

\newcommand{\norm}[1]{\Vert #1 \Vert}

\newcommand{\Ind}{{\bf I}}

\begin{document}

\maketitle

\maketitle
 \begin{abstract} In this paper we first show that for a locally compact amenable group $G$, 
 every proper abstract Segal algebra of the Fourier algebra on $G$ is not approximately amenable; consequently, every proper Segal algebra on a locally compact abelian group is not approximately amenable. Then  using the hypergroup generated by the dual of a compact group, it is shown that all  proper Segal algebras of a class of compact groups including the $2\times 2$ special unitary group, $SU(2)$, are not approximately amenable.

 \noindent{{\bf 2010 Mathematics subject classification:} primary 46H20; secondary 43A20, 43A62, 46H10, 46J10.
}\\
\end{abstract}

\begin{section}{Introduction}\label{s:introduction}
The notion of approximate amenability of a Banach algebra was introduced by
Ghahramani and Loy in \cite{aa}. 
A Banach algebra $\cal A$ is said to be {\it approximately amenable} if for every $\cal A$-bimodule
$X$ and every bounded derivation $D:{\cal A}\rightarrow X$, there exists a net $(D_\alpha)$ of inner derivations
such that 
\[
\lim_{\alpha} D_\alpha(a) = D(a)\ \ \text{ for all\ } a \in {\cal A}.
\] 
This is not the original definition but it is equivalent.
 In \cite{aa}, it is shown that approximately amenable algebras have approximate identities (possibly unbounded) and that approximate amenability is preserved where passing to quotient algebras and to closed ideals that have a bounded approximate identity. 

In this paper, we study the approximate amenability of proper abstract Segal algebras of Fourier algebras and subsequently, Segal algebras on abelian groups and compact groups.
The definition of Segal algebras will be given in Section~\ref{s:SA(G)-abelian}.
\vskip1.5em

Approximate amenability of Segal algebras has been studied in several papers. Dales and Loy, in \cite{da}, studied approximate amenability of Segal algebras on $\Bbb{T}$ and $\Bbb{R}$. They showed that certain Segal algebras on $\Bbb{T}$ and $\Bbb{R}$ are not approximately amenable. It was further conjectured that no proper Segal algebra on $\Bbb{T}$ is approximately amenable. Choi and Ghahramani, in \cite{ch}, have shown the stronger fact that no proper Segal algebra on $\Bbb{T}^{d}$ or $\Bbb{R}^{d}$ is approximately amenable. 
We extend the result of Choi and Ghahramani to apply to all locally compact abelian groups, not just $\Bbb{T}^d$ and $\Bbb{R}^d$. Our approach, like that of Choi-Ghahramani and Dales-Loy, is to apply the Fourier transform and work with abstract Segal subalgebras of the Fourier algebra of a locally compact abelian group.

In fact we prove a more general result in Section~\ref{s:SA(G)-abelian}: when $G$ is an amenable locally compact group, no proper Segal subalgebra of Fourier algebra is approximately amenable. The proof makes use of equivalence between amenability of $G$ and the so-called Leptin condition on $G$.

  In the rest of the paper, we try to apply similar tools, this time for compact groups.
In Section~\ref{s:hypergroup-approach} our idea is to view the dual of compact groups as a discrete hypergroup. The Fourier space of hypergroups studied as well. In Section~\ref{ss:Folner-condition-on-^G}, we introduce an analogue for hypergroups of the classical Leptin condition, and show that this holds for certain examples. Eventually, in Section~\ref{ss:on-S^1(^SU(2))}, We apply these tools to show that for a class of compact groups  including $SU(2)$, every proper Segal algebra is not approximately amenable.

\end{section}


\begin{section}{Abstract Segal algebras of Fourier algebra on amenable groups}\label{s:SA(G)-abelian}

Let $\cal A$ be a commutative Banach algebra. We denote by $\sigma(\cal A)$ the {\it spectrum} of ${\cal A}$ which is also called {\it maximal ideal space} or {\it character space} of ${\cal A}$.
A commutative Banach algebra $\cal A$ is called {\it regular}, if for every $\phi$   in   $\sigma({\cal A})$  and every $U$, open neighborhood of $\phi$ in the Gelfand topology, there exists an element $a\in {\cal A}$ such that $\phi(a)=1$ and $\psi(a)=0$ for each $\psi\in \sigma({\cal A})\setminus U$.  Using Gelfand representation theory, for a commutative semisimple regular Banach algebra $\cal A$, it can be viewed  as an algebra of continuous functions on its spectrum, $\sigma(\cal A)$. So for each $a\in \cal A$, $\supp (\hat{a})$ is defined  the support of function $\hat{a}$ as a subset of $\sigma(\cal A)$.



\begin{proposition}\label{p:regularity-of-asa}
Let ${\cal A}$ be a commutative semisimple regular Banach algebra and let $\cal B$ be an abstract Segal algebra of $\cal A$. Then $\cal B$ is also semisimple and regular. Moreover, $\cal B$ contains all elements $a\in\cal A$ such that $\supp(\hat{a})$ is compact.
\end{proposition}

\begin{proof}
By \cite[Theorem~2.1]{bu}, $\sigma({\cal B})$ is homeomorphic to $\sigma(\cal A)$ and ${\cal B}$ is semisimple. 
Theorem~3.6.15 and Theorem~3.7.1 of \cite{ric} imply that for  a commutative regular Banach algebra $\cal A$  and  a closed subset $E$  of $\sigma({\cal A})$ in the Gelfand spectrum topology, if $a\in {\cal A}$ such that $|\phi(a)|\geq\delta>0$ for every $\phi \in E$, then 
there exists some $a'\in {\cal A}$ such that $\phi(aa')=1$ for every $\phi\in E$. We call this property {\it local invertibility}. 

On the other hand, \cite[Proposition~2.1.14]{re} says that if $\cal A$ is a commutative semisimple algebra with local invertibility and $\cal I$ an ideal of $\cal A$ such that $\bigcap_{a\in \cal I}\ker \hat{a}=\emptyset$. Then $\cal I$ contains all elements $a\in {\cal A}$ such that  $\hat{a}$ has a compact support in $\sigma(\cal A)$. 

Now since ${\cal A}$ is regular, for a closed set $E\subseteq \sigma(\cal A)$ and $\phi\in\sigma({\cal A})\setminus E$, we have some $a\in \cal A$ such that $\hat{a}|_E\equiv 0$, $\hat{a}(\phi)=1$, and $\supp(\hat{a})$ is a compact subset of $\sigma(\cal A)$. Therefore $a\in\cal B$ implying that ${\cal B}$ is regular.

\end{proof}

Let $G$ be a locally compact group, equipped with a fixed left Haar measure $\lambda$.
The Fourier algebra of  $G$ was defined and studied by Eymard in \cite{ey}. In the following lemma, we summarize the main features of the Fourier algebra, denoted by $A(G)$, which we need here. We denote by $C_c(G)$ the space of continuous, compactly supported, complex-valued functions on $G$.

\begin{lemma}\label{l:A(G)-general-construction}
Let $G$ be a locally compact group and $K$ be a compact subset of $G$, and let $U$ be an open subset of $G$ such that $K \subset U$.
For each measurable set $V$  such that $0<\lambda(V)<\infty$ and $KVV^{-1} \subseteq U$, we can find $u_{V}\in A(G) \cap C_c(G)$ such that
\begin{enumerate}
\item{$u_{V}\big(G\big) \subseteq [0,1]$.}
\item{$u_{V}|_{K} \equiv 1$.}
\item{${ \supp}u_{V} \subseteq U$.}
\item{$\|u_{V}\|_{A(G)} \leq (\lambda(KV)/\lambda(V))^{1 \over 2}$.}
\end{enumerate}
\end{lemma}
\indent The existence of $f_{V} \in A(G) \cap C_{c}(G)$ satisfying $(1)$, $(2)$, and $(3)$ is proved in  \cite[Lemma 3.2]{ey}. Also based on the proof of  \cite[Lemma 3.2]{ey} and the definition of the norm of $A(G)$, one can show $(4)$. 

Since $\sigma(A(G))=G$, Lemma~\ref{l:A(G)-general-construction}  shows that $A(G)$ is a commutative semisimple  regular algebra.

We say that the Banach algebra $({\cal B},\|\cdot\|_{{\cal B}})$ is an {\it abstract Segal algebra} of a Banach algebra 
$({\cal A},\|\cdot\|_{\cal A})$ if \\
\begin{enumerate}
\item{ ${\cal B}$ is a dense left ideal in ${\cal A}$.}
\item{There exists $M>0$ such that $\|b\|_{\cal A} \leq M
\|b\|_{\cal B}$
for each $b\in {\cal B}$.}
\item{There exists $C>0$ such that $\|ab\|_{\cal B}\leq
C\|a\|_{\cal A}\|b\|_{\cal B}$ for all  $a,b \in {\cal B}$.}
\end{enumerate}

\noindent If ${\cal B}$ is a proper subalgebra of ${\cal A}$, we call it a {\it proper } abstract Segal algebra of ${\cal A}$.

\vskip1.5em

Let $G$ be a locally compact group. A linear subspace $S^1(G)$ of $L^1(G)$, the group algebra of $G$,  is said to be
a {\it Segal algebra} on $G$, if it satisfies the following conditions:\\
\begin{enumerate}
\item{$S^1(G)$ is dense in $L^1(G)$.}
\item{$S^1(G)$ is a Banach space under some norm $\norm{\cdot}_{S^1}$ and $\norm{ f }_{S^1}\geq  \norm{ f }_1$ for all $f \in S^1(G)$.}
\item{$S^1(G)$ is left translation invariant and the map $x\mapsto L_xf$ of $G$ into $S^1(G)$ is
continuous where $L_xf(y)=f(x^{-1}y)$.}
\item{$\norm{ L_xf }_{S^1} =\norm{f}_{S^1}$ for all $f \in S^1(G)$ and $x \in G$.}
\end{enumerate}

Note that every Segal algebra on $G$ is an abstract Segal algebra of  $L^1(G)$ with convolution product. Similarly, we call a Segal algebra on $G$ {\it proper} if it is a proper subalgebra of $L^1(G)$. For the sake of completeness we will give some examples of Segal algebras.
\begin{itemize}
\item{Let ${\cal L}A(G):=L^1(G)\cap A(G)$ and 
$|||h|||:=\|h\|_1+\|h\|_{A(G)}$ for $h\in{\cal L}A(G)$. Then
${\cal L}A(G)$ with norm $|||.|||$ is a Banach space;
this space was studied extensively by Ghahramani
and Lau in \cite{gl}. They have shown that ${\cal L}A(G)$ with the
convolution product is a Banach algebra called  {\it Lebesgue-Fourier algebra} of
$G$; moreover, it is a Segal algebra on locally compact group $G$. ${\cal L}A(G)$ is a proper Segal algebra on $G$ if and only if $G$ is not discrete.

Also, ${\cal L}A(G)$ with pointwise multiplication is a Banach algebra and even
an abstract Segal algebra of $A(G)$. 
Similarly, ${\cal L}A(G)$ is a proper subset of $A(G)$ if and only if $G$ is not compact.   }
\item{The convolution algebra $L^1(G)\cap L^p(G)$ for $1\leq p <\infty$ equipped with the norm $\norm{f}_{1}+\norm{f}_{p}$ is a Segal algebra.}
\item{Similarly, $L^1(G) \cap C_0(G)$  with respect to the norm $\norm{f}_{1}+\norm{f}_{\infty}$ is a Segal algebra where $C_0(G)$ is the $C^*$-algebra of continuous functions vanishing at infinity.}
\item{ Let $G$ be a  compact group, $\cal F$ denote the Fourier transform, and ${\cal L}^p(\widehat{G})$ be the space which will be defined in (\ref{eq:cal L^p(^G)}).  We can see that ${\cal F}^{-1}\big({\cal L}^p(\widehat{G})\big)$, we denote by $\Cp$, equipped with convolution  is  a subalgebra of $L^1(G)$. For $\norm{f}_\Cp:=\norm{{\cal F}f}_\Lp$,
one can show that  for each $1\leq p \leq 2$, $\big(\Cp,\norm{\cdot}_{\Cp}\big)$ is a Segal algebra of $G$.}
\end{itemize}

In \cite{ch}, a nice criterion is developed to prove the non-approximate amenability of Banach algebras. At several points below we will rely crucially on that criterion.
For this reason, we present a version of that below.
Recall that for a Banach algebra $\cal A$, a sequence  $(a_n)_{n\in \Bbb{N}}\subseteq \cal A$ is called {\it multiplier-bounded} if $\sup_{n\in\Bbb{N}}\norm{a_n b}\leq M \norm{b}$ for all $b\in \cal A$.

\begin{theorem}\label{t:SUM}
Let $\cal A$ be a Banach algebra. Suppose that there exists an unbounded but multiplier-bounded sequence $(a_n)_{n\geq 1}\subseteq {\cal A}$ such that 
\[
a_na_{n+1}=a_n=a_{n+1}a_n
\]
 for all $n$. Then $\cal A$ is not approximately amenable.
\end{theorem} 

The following theorem is the main theorem of this section.

\begin{theorem}\label{t:SA(G)-non-a-a-abelian}
Let $G$ be a locally compact amenable group and $SA(G)$ a proper abstract Segal algebra of $A(G)$. Then $SA(G)$ is not approximately amenable.
\end{theorem}

\begin{proof} 
Since  $SA(G)$ is a proper abstract Segal algebra of $A(G)$, the norms $\norm{\cdot}_{SA(G)}$ and
$\norm{\cdot}_{A(G)}$ can not be equivalent. On the other hand,  $M \norm{f}_{A(G)}\leq  \norm{f}_{SA(G)}$ for $f\in SA(G)$ and  some $M>0$. Therefore, we can find a sequence $(f_n)_{n\in \Bbb{N}}$ in $C_c(G)\cap A(G)$, and hence by Proposition~\ref{p:regularity-of-asa}  in $SA(G)$, such that 
\begin{equation}\label{eq:condition-of-f_n}
n \norm{f_n}_{A(G)} \leq \norm{f_n}_{SA(G)}\ \ \text{for all $n\in \Bbb{N}$}.
\end{equation}

If $G$ is a  locally compact group, then $G$ is amenable if and only if it satisfies the {\it Leptin condition} i.e. for every $\epsilon>0$ and compact set $K\subseteq  G$, there
exists a relatively compact neighborhood $V$ of $e$ such that  $\lambda(KV)/\lambda(V) <1+\epsilon$, \cite[Section~2.7]{pi}. 

Fix $D>1$. Using the Leptin condition and Lemma~\ref{l:A(G)-general-construction}, we can generate a sequence $(u_n)_{n\in \Bbb{N}}$ inductively in $A(G)\cap C_c(G)\subseteq SA(G)$  such that $u_n|_{\supp f_n} \equiv 1$,
 ${ \supp}u_{n} \subseteq \{x\in G: u_{n+1}(x)=1\}$, and
 $\|u_n\|_{A(G)} \leq  D$.
Hence  $u_{n}f_n=f_n$ and $u_n u_{n+1}=u_n$ for every $n\in \Bbb{N}$.

 So here we only need to prove unboundedness of $(u_n)_{n\geq 1}$ in $\norm{\cdot}_{SA(G)}$.
Suppose otherwise  then $\sup_{n\in\Bbb{N}}\norm{u_n}_{SA(G)}=C'$
for some $0<C' <\infty$. Then for each $n\in \Bbb{N}$, one can write
\[
\norm{f_n}_{SA(G)}=\norm{u_{n} f_n}_{SA(G)}\leq C \norm{f_n}_{A(G)} \norm{u_{n}}_{SA(G)} \leq C'C \norm{f_n}_{A(G)}
\]
for some fixed $C>0$. But this violates the condition (\ref{eq:condition-of-f_n}); therefore, $(u_n)_{n\in\Bbb{N}}$  is unbounded in $\norm{\cdot}_{SA(G)}$. Consequently, Theorem~\ref{t:SUM} shows that $SA(G)$ is not approximately amenable.
\end{proof}

We should recall that Leptin condition played a crucial role in the proof. Indeed we have used Leptin condition to impose $\norm{\cdot}_{A(G)}$-boundedness to the sequence $(u_n)$. 
As we mentioned before, the approximate amenability of all proper Segal algebras on $\Bbb{R}^d$ has been studied by Choi and Ghahramani in~\cite{ch}. We are therefore motivated to conclude a generalization of their result in the following corollary.

\begin{cor}\label{c:non-a-a-of-proper-segal-algebras-of-abelian-groups}
Let $G$ be a locally compact abelian group. Then every proper Segal subalgebra of $L^1(G)$ is not approximately amenable. 
\end{cor} 

\begin{proof} Let $S^1(G)$ be a proper Segal algebra on $G$. Applying the Fourier transform on $L^1(G)$, we may transform $S^1(G)$ to a proper abstract Segal algebra $SA(\widehat{G})$ of $A(\widehat{G})$, the Fourier algebra on the dual of group $G$. By Theorem~\ref{t:SA(G)-non-a-a-abelian}, $SA(\widehat{G})$ is  not approximately amenable.
\end{proof}
 
\begin{rem}\label{eg:lebesgue-Fourier-algebra}
In particular, by Theorem~\ref{t:SA(G)-non-a-a-abelian}, for amenable locally compact group $G$  the Lebesgue Fourier algebra with pointwise multiplication is approximately amenable if and only if $G$ is a compact group, in which case it equals the Fourier algebra of $G$.
Moreover, by Corollary~\ref{c:non-a-a-of-proper-segal-algebras-of-abelian-groups},
for  a locally compact abelian group $G$, the Lebesgue Fourier algebra with convolution product is approximately amenable if and only if $G$ is discrete, in which case it equals $\ell^1(G)$.
\end{rem}

 \vskip2.0em
\end{section}


\begin{section}{Hypergroups and their Fourier algebra}\label{s:hypergroup-approach}

Although the dual of a compact group is not a group, in general, it is a (commutative discrete) hypergroup. We give the background needed for this result in Subsection~\ref{ss:dual-of-compact}. 
Muruganandam, \cite{mu1}, gave a definition of the {\it Fourier space}, $A(H)$, of a hypergroup $H$ and showed
that $A(H)$ is a Banach algebra with pointwise product for certain commutative hypergroups.
 In Subsection~\ref{ss:Fourier-on-^G}, we study the Fourier space on the dual of a compact group $G$, denoted by $A(\widehat{G})$. We show that indeed for each compact group $G$, $A(\widehat{G})$ is a Banach algebra. 
 

\begin{subsection}{Preliminaries and notations}\label{ss:dual-of-compact}
For studying hypergroups, we mainly rely on \cite{bl}. As a short summary for hypergroups we give the following definitions and facts. Let  $(H,*,\tilde{ }\;)$ be a (locally compact) {\it hypergroup} possessing a Haar measure $h$. The notation $A*B$ stands for  
\[
\bigcup\{\supp(\delta_x*\delta_y):\; \text{for all }\;x\in A, y\in B\}
\]
 for $A,B$ subsets of the hypergroup $H$. With abuse of notation, we use $x*A$ to imply $\{x\}*A$. 
Let $C_c(H)$ be the space of all complex valued compact supported continuous functions over $H$. We define
\[
L_xf(y)=\int_H f(t) d\delta_x*\delta_y(t) \ \ f\in C_c(H),\ x,y\in H.
\]

\noindent Defining 
\[
f*_hg(x):=\int_{H} f(t) L_{\tilde{t}}g(x) dh(t), \ \ \tilde{f}(x):=f(\tilde{x}),\ \text{ and}\ \ f^*:=\overline{(\tilde{f})},
\]
 we can see that all of the functions $f:H\rightarrow \Bbb{C}$ such that 
\[
\norm{f}_{L^1(H,h)}:=\int_{H} |f(t)| dh(t)<\infty
\]
 form a Banach $*$-algebra, denoted by $(L^1(H,h),*_h,\norm{\cdot}_h)$;  it is called the {\it hypergroup algebra} of $H$.

\vskip1.5em

\noindent If $H$ is discrete and $h(e)=1$, we have
\[
h(x)=(\delta_{\tilde{x}}*\delta_x(e))^{-1}.
\]

\begin{lemma}\label{l:Haar-convolution-of-dirac-functions}
Let $H$ be a discrete hypergroup. For each pair $x,y\in H$, 
\[
\delta_x *_h \delta_y(z) = \delta_x * \delta_y(z)\frac{h(x)h(y)}{h(z)}
\]
for each $z\in H$.
\end{lemma}


\vskip1.5em

In this section, let $G$ be a compact group and $\widehat{G}$ the set of all irreducible unitary representations of $G$. In this paper we follow the notation of \cite{du} for the dual of compact groups.  Where ${\cal H}_\pi$ is the finite dimensional Hilbert space related to the representation $\pi\in\widehat{G}$, we define $\chi_\pi:=Tr\pi$, the group character generated by $\pi$ and $d_\pi$ denotes the dimension of  ${\cal H}_\pi$. Let $\phi=\{\phi_\pi:\; \pi\in\widehat{G}\}$ if $\phi_\pi\in{\cal B}({\cal H}_\pi)$ for each $\pi$ and define 
\[
\norm{\phi}_{{\cal L}^\infty(\widehat{G})}:=\sup_\pi\norm{\phi_\pi}_\infty \]
 for $\norm{\cdot}_\infty$, the operator norm. The set of all those $\phi$'s  with $\norm{\phi}_{{\cal L}^\infty(\widehat{G})}<\infty$ forms a $C^*$-algebra; we denote it by ${{\cal L}^\infty(\widehat{G})}$. It is well known that ${\cal L}^\infty(\widehat{G})$ is isomorphic to the von Neumann algebra of $G$ i.e. the dual of $A(G)$, see \cite[8.4.17]{du}. We define
\begin{equation}\label{eq:cal L^p(^G)}
{\cal L}^p(\widehat{G})=\{\phi\in{\cal L}^\infty(\widehat{G}): \norm{\phi}_{{\cal L}^p(\widehat{G})}^p:=\sum_{\pi \in\widehat{G}} d_\pi \norm{\phi_\pi}_p^p<\infty\},
\end{equation}
for $\norm{\cdot}_p$, the $p$-Schatten norm. For each $p$, ${\cal L}^p(\widehat{G})$ is an ideal of ${\cal L}^\infty(\widehat{G})$, see \cite[8.3]{du}. 
Moreover, we define
 \[
 {\cal C}_0(\widehat{G})=\{\phi\in{\cal L}^\infty(\widehat{G}):\;\lim_{\pi\rightarrow\infty}   \norm{\phi_\pi}_\infty=0\}.\]
For each $f\in L^1(G)$, ${\cal F}(f)=(\hat{f}(\pi))_{\pi\in\widehat{G}}$ belongs to ${\cal C}_0(\widehat{G})$, where $\cal F$ denotes Fourier transform and
\[
\hat{f}(\pi)=\int_G f(x)  \pi(x^{-1}) dx.
\]
 Indeed, ${\cal F}(L^1(G))$ is a dense subset of ${\cal C}_0(\widehat{G})$ and ${\cal F}$ is an isomorphism from Banach algebra $L^1(G)$ onto its image.

  For each two irreducible representations $\pi_{1},\pi_{2}\in\widehat{G}$, we know that $\pi_1 \otimes \pi_2$ can be written as a product of $\pi'_1,\cdots,\pi'_n$ elements of $\widehat{G}$ with respective multiplicities  $m_1,\cdots,m_n$, i.e.
\[
\pi_1\otimes \pi_2 \cong \bigoplus_{i=1}^n m_i \pi'_i.
\]
We define a convolution on  $\ell^1(\widehat{G})$ by
\begin{equation}\label{eq:hypergroup-convolution-on-^G}
\delta_{\pi_1}* \delta_{\pi_2}:=\sum_{i=1}^n \frac{m_i d_{\pi'_i}}{d_{\pi_1}d_{\pi_2}}\delta_{\pi'_i}
\end{equation}
and define an involution  by $ \tilde{\pi}=\overline{\pi}$
for all $\pi,\pi_1,\pi_2\in \widehat{G}$. It is straightforward to verify that $(\widehat{G}, * ,\tilde{ }\;)$ forms a discrete commutative hypergroup such that $\pi_0$, the trivial representation of $G$, is the identity element of $\widehat{G}$
and $h(\pi)=d_\pi^2$ is the Haar measure of $\widehat{G}$. 

\vskip2.0em

\begin{eg}\label{eg:SU(2)}
Let $\widehat{SU(2)}$ be the hypergroup of all irreducible representations of the compact group $SU(2)$. We know that 
\[
\widehat{SU(2)}=(\pi_\ell)_{\ell \in 0,\frac{1}{2},1,\frac{3}{2}, \cdots}
\]
where the dimension of $\pi_\ell$ is $2\ell+1$, see \cite[29.13]{he2}. Moreover, 
\[
\pi_\ell \oplus \pi_{\ell'}= \bigoplus_{r=|\ell-\ell'|}^{\ell+\ell'} \pi_r = \pi_{|\ell-\ell'|} \oplus \pi_{|\ell-\ell'|+1 } \oplus \cdots \oplus \pi_{\ell+\ell'} \ \ \ \text{(by \cite[Theorem 29.26]{he2})}.\]
So using Definition~\ref{eq:hypergroup-convolution-on-^G}, we have that
\[
\delta_{\pi_\ell}*\delta_{\pi_{\ell'}} = \sum_{r=|\ell-\ell'|}^{\ell+\ell'} \frac{(2r+1)}{(2\ell+1)(2\ell'+1)}\delta_{\pi_r}.\]
Also
$
\tilde{\pi_\ell}=\pi_\ell
$ and $h(\pi_\ell)=(2\ell+1)^2$ for all $\ell$.
\end{eg}

\vskip2.0em

\begin{eg}\label{eg:product-of-finite-groups}

Suppose that $\{G_i\}_{i\in\Ind}$ is a non-empty family of compact groups for arbitrary indexing set $\Ind$. Let  $G:=\prod_{i\in\Ind}G_i$ be the product of $\{G_i\}_{i\in \Ind}$ i.e. $
G = \{(x_i)_{i\in \Ind}:\ x_i\in G_i\}$
equipped with product topology. Then $G$ is a compact group and by \cite[Theorem~27.43]{he2},
\[
\widehat{G}=\{\pi=\bigotimes_{i\in\Ind}\pi_i:\ \text{such that}\ \pi_i\in\widehat{G}_i \ \text{and}\ \pi_i=\pi_0\ \text{except for finitely many $i\in\Ind$}\}
\] 
equipped with the discrete topology. Moreover, for each $\pi=\bigotimes_{i\in\Ind}\pi_i \in \widehat{G}$, $d_\pi=\prod_{i\in\Ind} d_{\pi_i}$. 

\noindent If $\pi_k=\bigotimes_{i\in\Ind} \pi_i^{(k)} \in \widehat{G}$ for $k=1,2$, one can show that
\[
\delta_{\pi_1}* \delta_{\pi_2}(\pi)= \prod_{i\in \Ind} \delta_{\pi_i^{(1)}} *_{\widehat{G}_i} \delta_{\pi_i^{(2)}}(\pi_i)\ \ \ \ \text{for}\   \pi=\bigotimes_{i\in\Ind}\pi_i \in \widehat{G},
\]
where $*_{\widehat{G}_i}$ is the hypergroup product in $\widehat{G}_i$ for each $i\in\Ind$. Also, each character $\chi$ of $G$ corresponds to a family of characters $(\chi_i)_{i\in\Ind}$ such that $\chi_i$ is a character of $G_i$ and $\chi(x)=\prod_{i\in\Ind}\chi_i(x_i)$
for each $x=(x_i)_{i\in\Ind}\in G$. Note that $\chi_1\equiv 1$ for all of $i\in\Ind$ except finitely many.

\end{eg}

\end{subsection}

\begin{subsection}{The Fourier algebra of the dual of  a compact group}\label{ss:Fourier-on-^G}

For a compact hypergroup $H$, first Vrem in \cite{vr} defined the {\it Fourier space} similar to the Fourier algebra of a compact group. Subsequently, Muruganandam, \cite{mu1}, defined the {\it Fourier-Stieltjes space} on an arbitrary (not necessary compact) hypergroup $H$ using irreducible representations of $H$   analogous  to the Fourier-Stieltjes algebra on locally compact groups.  Subsequently, he defines 
the { Fourier space} of hypergroup $H$, as a closed subspace of the Fourier-Stieltjes algebra, generated by $\{f*_h\tilde{f}:\; f\in L^2(H,h)\}$. Also Muruganandam shows that where $H$ is commutative $A(H)$ is 
$ \{f*_h\tilde{g}:\; f,g\in L^2(H,h)\}$
 and  $\norm{u}_{A(H)}=\inf  \norm{f}_2 \norm{g}_2$ for all $f,g\in L^2(H,h)$ such that  $u=f*\tilde{g}$. 
He calls hypergroup $H$ a {\it regular Fourier hypergroup}, if the Banach space $({A}(H),\norm{\cdot}_{{A}(H)})$ equipped with pointwise product is a Banach algebra.

We prove a hypergroup version of Lemma~\ref{l:A(G)-general-construction} which shows some important properties of the Banach space $A(H)$ for an arbitrary hypergroup $H$ (not necessarily a regular Fourier hypergroup) and  $\widehat{G}$ is a regular Fourier hypergroup. Some parts of the following Lemma have already been shown in \cite{vr} for compact
hypergroups and that proof is applicable to general hypergroups. Here we present a complete proof for the lemma.

\begin{lemma}\label{l:A(H)-properties}
Let $H$ be a hypergroup, $K$ a compact subset  of $H$ and $U$ an open subset of $H$ such that $K\subset U$. Then for each measurable set $V$  such that $0<h_H(V)<\infty$  and $\overline{K *V*\tilde{V}} \subseteq U$, there exists some $u_V\in A(H) \cap C_c(H)$ such that:\\
\begin{enumerate}
\item{$u_V(H)\geq 0$.}
\item{$u_V|_K=1$.}
\item{$\supp(u_V) \subseteq U$.}
\item{$\norm{u_V}_{A(H)} \leq \big( h_H(K*V)/{h_H(V)}\big)^{\frac{1}{2}}$.}
\end{enumerate}
\end{lemma}

\begin{proof}
Let us define
\[
u_V:=\frac{1}{h_H(V)} 1_{K*V} *_h \tilde{1}_{V}.
\]
Clearly $u_V\geq 0$. Moreover, for each $x\in K$ , 
\begin{eqnarray*}
 h_H(V) u_V(x) &=&   1_{K*V} *_h \tilde{1}_{V}(x)\\
&=& \int_{ H} 1_{K*V}(t) L_{\tilde{t}}\tilde{1}_V(x) dh_H(t) \\
&=&  \int_{ H} 1_{K*V}(t) L_{\tilde{x}} {1}_V(t) dh_H(t)\\
&=&  \int_{t\in H} L_{x}1_{K*V}(t)  {1}_V(t) dh_H(t)\ \ \ \text{(by \cite[Theorem~1.3.21]{bl})}\\
&=&  \int_{ V}   \langle 1_{K*V} , \delta_{{x}}*\delta_{t} \rangle dh_H(t)\\
&=& {h_H(V)}.
\end{eqnarray*}

\noindent Also \cite[Proposition~1.2.12]{bl} implies that 
 \[
 \supp( 1_{K*V} *_h \tilde{1}_{V})\subseteq \overline{\left(K*V *\tilde{V}\right)}\subseteq U.
 \]
Finally, by \cite[Proposition~2.8]{mu1}, we know that
\begin{eqnarray*}
\norm{u_V}_{A(H)}&\leq & \frac{\norm{1_{K*V}}_2 \norm{1_V}_2}{h_H(V)}= \frac{h_H(K*V)^{\frac{1}{2}} h_H(1_V)^{\frac{1}{2}}}{h_H(V)}= \frac{ h_H(K*V)^{\frac{1}{2}}}{h_H(V){\frac{1}{2}}}.
\end{eqnarray*}
\end{proof}

\begin{rem}\label{r:existence-of-the-V}
For each pair $K,U$ such that $K \subset U$,  we can always find  a relatively compact neighborhood  $V$ of $e_H$ that satisfies the conditions in Lemma~\ref{l:A(H)-properties}. But the proof is quite long and in our application the existence of such $V$ will be clear.
\end{rem}


Given a commutative hypergroup, it is not immediate that it is a regular Fourier hypergroup or not.
We will show that when $G$ is a compact group, the hypergroup $\widehat{G}$ is a regular Fourier hypergroup.
\begin{dfn}
For $A(G)$, Fourier algebra on $G$, we define 
\[
ZA(G):=\{f\in A(G): f(yxy^{-1})=f(x)\ \text{for all $x\in G$}\}.
\]
which is a Banach algebra with pointwise product and $\norm{\cdot}_{A(G)}$.
\end{dfn}

\begin{theorem}\label{t:Fourier-of-^G}
Let $G$ be a compact group. Then $\widehat{G}$ is a regular Fourier hypergroup and $A(\widehat{G})$ is isometrically  isomorphic with the center of the group algebra $G$, i.e. $A(\widehat{G})\cong ZL^1(G)$.
Moreover, the hypergroup algebra of $\widehat{G}$, $L^1(\widehat{G},h)$, is isometrically isomorphic with $ZA(G)$.
\end{theorem}

\begin{proof}
Let $\cal F$ be the Fourier transform on $L^1(G)$.
We know that ${\cal F}|_{L^2(G)}$ is an isometric isomorphism from $L^2(G)$ onto $ {\cal L}^2(\widehat{G})$.
By the properties of the Fourier transform, \cite[Proposition~4.2]{du}, for each $f\in ZL^2(G)$ and $g\in L^1(G)$ we have
\begin{equation}\label{eq:ZL^2-commutes-with-everything}
{\cal F}(f)\circ {\cal F}(g)={\cal F}(f*g)={\cal F}(g*f)={\cal F}(g)\circ {\cal F}(f).
\end{equation}
So ${\cal F}(f)$ commutes with all elements of ${\cal C}_0(\widehat{G})$; 	therefore, ${\cal F}(f)=(\alpha_\pi I_{d_\pi \times d_\pi})_{\pi\in\widehat{G}}$ for a family of scalars $(\alpha_\pi)_{\pi \in \widehat{G}}$ in $\Bbb{C}$.
Hence,
\begin{eqnarray*}
\norm{{\cal F}f}_2^2 = \sum_{\pi\in\widehat{G}} d_\pi \norm{\widehat{f}(\pi)}_2^2 = \sum_{\pi\in\widehat{G}} d_\pi \alpha_\pi^2 \norm{I_{d_\pi\times d_\pi}}_2^2 = \sum_{\pi\in\widehat{G}}  \alpha_\pi^2 {d_\pi}^2 
= \sum_{\pi\in\widehat{G}}  \alpha_\pi^2 h(\pi).
\end{eqnarray*}

\noindent Using the preceding identity, we define $
{\cal T}: ZL^2(G) \rightarrow L^2(\widehat{G},h)$ by
 ${\cal T}(f)=(\alpha_\pi)_{\pi\in\widehat{G}}$. 
Note that $\{\chi_\pi\}_{\pi\in\widehat{G}}$ forms an orthonormal basis for $ZL^2(G)$. Since ${\cal F}(\chi_{\tilde{\pi}})=d_{\tilde{\pi}}^{-1} I_{d_{\tilde{\pi}\times {\tilde{\pi}}}}$, ${\cal T}(\tilde{f})={\cal T}(f\tilde{)}$
 for each $f\in L^2(G)$ where  $\tilde{f}(x)=f(x^{-1})$. 
So ${\cal T}$ is an isometric isomorphism from $ZL^2(G)$ onto $L^2(\widehat{G},h)$.

 We claim that ${\cal T}(f \tilde{g})={\cal T}(f) *_h {\cal T}(g\tilde{)}$
 for all $f,g\in ZL^2(G)$. 
To prove our claim it is enough to show that ${\cal T}(\chi_{\pi_1} \chi_{\pi_2})={\cal T}(\chi_{\pi_1}) *_h {\cal T}(\chi_{\pi_2})$ for  $\pi_1,\pi_2\in \widehat{G}$.  
 Therefore, using Lemma~\ref{l:Haar-convolution-of-dirac-functions}, for each two representations $\pi_1,\pi_2\in\widehat{G}$, we have
\begin{eqnarray*}
{\cal T}(\chi_{\pi_1}\chi_{\pi_2}) &=& {\cal T}\left( 
\sum_{i=1}^n m_i \chi_{\pi'_i}\right)\\
&=& \sum_{i=1}^n m_i {\cal T}(\chi_{\pi'_i})\\
&=& \sum_{i=1}^n m_i d_{\pi'_i}^{-1} \delta_{\pi'_i}\\
&=& d_{\pi_1}^{-1} \delta_{\pi_1} *_h  d_{\pi_2}^{-1}\delta_{\pi_2}\\
&=& {\cal T}(\chi_{\pi_1})*_h {\cal T}(\chi_{\pi_2}).
\end{eqnarray*}
\vskip1.0em

\noindent  Now we can define a surjective extension $
 {\cal T}:ZL^1(G) \rightarrow {A}(\widehat{G})$, using the fact that $lin\{\chi_\pi\}_{\pi\in\widehat{G}}$ is dense in $ZL^1(G)$ as well and  $\norm{f}_1=\inf \norm{g_1}_2\norm{g_2}_2$ for all $g_1,g_2\in L^2(G)$ such that $f=g_1\tilde{g_2}$. 
Using the definition of the norm of $A(\widehat{G})$, 
   $\norm{{\cal T}(f)}_{{A}(\widehat{G})}=\norm{f}_1$ for each $f\in ZL^1(G)$. To show that the extension of ${\cal T}$ is onto, for each pair  $g_1,g_2\in ZL^2(G)$, we note that $g_1\tilde{g_2}\in ZL^1(G)$.
So, ${\cal T}$ is an isometric isomorphism. This implies that ${A}(\widehat{G})$ is a Banach algebra with pointwise product and hence $\big({A}(\widehat{G}),\cdot,\norm{\cdot}_{{A}(\widehat{G})}\big)\cong \big(ZL^1({G}),*,\norm{\cdot}_{1}\big)$. 

\vskip1.5em

The second part is similar to the first part of the proof. This time we consider the restriction 
 of the Fourier transform from $ZA(G)$ onto $ {\cal L}^1(\widehat{G})$.
 Again by an argument similar to (\ref{eq:ZL^2-commutes-with-everything}), we define an isometric mapping $\cal T'$ from $ZA(G)$ onto $L^1(\widehat{G},h)$. Since $lin\{\chi_\pi\}_{\pi\in \widehat{G}}$ is dense in $ZA(G)$, we observe that  ${\cal T}'$ is an isometric isomorphism from $ZA(G)$ as an $*$-algebra with complex conjugate and pointwise product onto $L^1(\widehat{G},h)$ as an $*$-algebra with $*_h$ convolution.
\end{proof}

\vskip2em


\end{subsection}


\begin{section}{The Leptin condition on Hypergroups}\label{ss:Folner-condition-on-^G}

In the proof of Theorem~\ref{t:SA(G)-non-a-a-abelian}, we used the Leptin condition for amenable groups. In this subsection we study the Leptin condition for the dual of hypergroups. In \cite{sk}, the Reiter condition was introduced for amenable hypergroups. Although the Reiter condition on hypergroups is defined similar to amenable groups,   the Leptin condition is a problem for hypergroups. There are some attempts to answer to this question for some special hypergroups in \cite{la4}. Recall that for each two subsets $A$ and $B$ of $X$, we denote the set $(A\setminus B) \cup (B\setminus A)$ by $A\bigtriangleup B$.

\begin{dfn}\label{d:Folner-Leptin-condition}
Let $H$ be a hypergroup. { We say that $H$ satisfies the {\it Leptin condition} if for every  compact subset $K$ of $H$ and $\epsilon>0$, there exists a measurable set $V$ in $H$ such that $0<h(V)<\infty$ and $h(K*V)/h(V) < 1+\epsilon$.}
\end{dfn}

We will use the Leptin condition, in the case where $H$ is the dual of a compact group $G$, to study approximate amenability for Segal algebras on $G$.

\begin{rem}\label{r:relatively-compact}
In the definitions of the conditions mentioned above, we can suppose that $V$ is a compact measurable set. To show this fact suppose that $H$ satisfies the Leptin condition. For compact subset $K$ of $H$ and $\epsilon>0$, there exists a measurable set $V$ such that  $h(K*V)/h(V) < 1 + \epsilon$. Using regularity of $h$, we can find compact set $V_1\subseteq V$ such that $h(V\setminus V_1) < h(V)/n$ for some positive integer $n>0$. It implies that $0< h(V_1)$ and  $ h(V)/h(V_1) < {n}/{(n-1)}$. Therefore,
\[
\frac{h(K*V_1)}{h(V_1)} \leq \frac{h(V)}{h(V_1)} \left( \frac{h(K*V_1)}{h(V)}\right)< \frac{n}{n-1}(1+\epsilon).
\]
So we can add compactness of $V$ to the definition of the Leptin condition. 
\end{rem}

 Note that since the duals of compact groups are commutative, they are all amenable hypergroups, \cite{sk}, but as we mentioned it does not say anything about the Leptin condition on those hypergroups. So the next question is for which compact groups $G$ do the hypergroups $\widehat{G}$ satisfy those conditions.   We will now show that Examples~\ref{eg:SU(2)} and \ref{eg:product-of-finite-groups} sometimes satisfy these conditions.

\begin{proposition}\label{p:Leptin-SU(2)^}
The hypergroup $\widehat{SU(2)}$ satisfies the Leptin condition.
\end{proposition}

\begin{proof} 
Given compact subset $K$ of $\widehat{SU(2)}$ and $\epsilon>0$.
Let $K$ and $\epsilon>0$ are given. $k:=\sup\{\ell:\; \pi_\ell\in K\}$. We select $m \geq k$ such that  for $V=\{\pi_\ell\}_{\ell=0}^m$, 
\begin{eqnarray}\label{eq:proof-of-Leptin}
\frac{h(\pi_k *V)}{h(V)} &=&  \frac{ \sum_{\ell=1}^{2m+2k+1}\ell^2 }{\sum_{\ell=1}^{2m+1} \ell^2}\\
&=& \frac{\frac{1}{3}(2m+2k+1)^3 + \frac{1}{2}(2m+2k+1)^2 + \frac{1}{6}(2m+2k+1)}{\frac{1}{3}(2m+1)^3 + \frac{1}{2}(2m+1)^2  + \frac{1}{6}(2m+1)}< 1+ \epsilon.\nonumber
\end{eqnarray}
But also for each $x\in K$, $x*V\subseteq \pi_k*V$. 
So using (\ref{eq:proof-of-Leptin}),
\[
\frac{h(K *V)}{h(V)} = \frac{h(\pi_k *V)}{h(V)} < 1+ \epsilon.
\]
\end{proof}

\begin{proposition}\label{p:Leptin-for-Prod-of-finite-groups}
Let $\{G_i\}_{i\in \Ind}$ be a family of compact groups whose duals have the Leptin condition and $G=\prod_{i\in\Ind}G_i$ is their product equipped with product topology. Then $\widehat{G}$ satisfies the Leptin condition.
\end{proposition}

\begin{proof}
Let $K \subseteq \widehat{G}$ be an arbitrary compact subset. Then, there exists some  $F\subseteq \Ind$ finite such that $K\subseteq \bigotimes_{i\in F} K_i \otimes E_F^c$ where $K_i$ is a compact subset of $\widehat{G}_i$ and $E_F^c=\bigotimes_{i\in\Ind\setminus F} \pi_0$ where $\pi_{0}$'s are the trivial representations of the corresponding $\widehat{G}_i$. 

Using the Leptin condition for each $G_i$, there exists some compact  set $V_i$ which satisfies the Leptin condition for $K_i$ and $\epsilon>0$   i.e. $h_{G_i}(K_i*V_i)<(1+ \epsilon) h_{G_i}(V_i)$. Therefore, for the compact set $V=(\bigotimes_{i\in F} V_i )\otimes E_F^c$,
\[
\frac{h(K*V)}{h(V)} \leq \prod_{i\in F} \frac{ h_{G_i}(K_i*V_i)}{h_{G_i}(V_i)} < (1+\epsilon)^{|F|}.
\]




\end{proof} 

\begin{rem}
If $G$ is finite then $\widehat{G}$ satisfies the Leptin condition; hence, for a family of finite groups say $\{G_i\}_{i\in\Ind}$, $G:=\prod_{i\in\Ind}G_i$, $\widehat{G}$ satisfies the Leptin condition. 
\end{rem}

\end{section}

\begin{section}{Segal algebras on compact groups whose duals satisfy Leptin condition}\label{ss:on-S^1(^SU(2))}

 In this section, we apply hypergroup approach to the original questions for Segal algebras. We show that every proper Segal algebra on ${G}$ is not approximately amenable if $G$ is a compact group if $\widehat{G}$ satisfies the Leptin condition.  First we need a general lemma for Banach algebras.

\begin{lemma}\label{l:characters-in-ZS^1(G)}
Let ${\cal A}$ be a Banach algebra and ${\cal J}$ be a dense left  ideal of ${\cal A}$. Then for each idempotent element $p$ in the center of  algebra ${\cal A}$  i.e. $p^2=p \in Z({\cal A})$, $p$ belongs to ${\cal J}$.
\end{lemma}

\begin{proof}
 Since $\cal J$ is dense in $\cal A$, there exists an element $a \in \cal J$ such that $
\norm{ p - a}_{\cal A} <1$.
Let us define
\[
b:= p  + \sum_{n=1}^\infty (p - a)^{n}.
\]
One can check that $p  b -    p  b ( p - a) = p  b a$,
which is an element in $\cal J$.
On the other hand,
\begin{eqnarray*}
p  b -  p  b ( p - a) &=&  p  \left( p + \sum_{n=1}^\infty(p -a)^{ n}\right) - p \left( p + \sum_{n=1}^\infty  (p -a)^{ n} \right)   (p -a)\\
&=&  p +  p \sum_{n=1}^\infty(p -a)^{ n} - p \sum_{n=2}^\infty  (p -a)^{ n}- p (p -a)\\
&=&  p +  p (p -a) - p (p -a) =p.
\end{eqnarray*}
\end{proof}


  \vskip2.0em
The main theorem of this section is as follows.

\begin{theorem}\label{t:Segal-of-S^1(G)-G-Leptin}
Let $G$ be a compact group such that $\widehat{G}$ satisfies the  Leptin condition. Then every proper Segal algebra on $G$ is not approximately amenable.
\end{theorem}

\begin{proof}
Let $S^1(G)$ be a proper Segal algebra on $G$.
By Lemma~\ref{l:characters-in-ZS^1(G)}, $S^1(G)$ contains all central idempotents $d_\pi\chi_\pi$ for each $\pi\in \widehat{G}$.
If $\cal T$ is the map defined in the proof of Theorem~\ref{t:Fourier-of-^G}, then ${\cal T}^{-1}(\delta_\pi)=d_\pi\chi_\pi\in S^1(G)$. So in order to use Theorem~\ref{t:SUM}, we look for a suitable sequence in $A(\widehat{G})$ with compact supports.
\vskip1.0em

Fix $D>1$. Using the Leptin condition on $\widehat{G}$,  for every arbitrary non-void compact set $K$ in $\widehat{G}$, we can find a finite subset $V_K$ of $\widehat{G}$ such that $h(K*V_K)/h(V_K)<D^{2}$. Using Lemma~\ref{l:A(H)-properties}  for \[
v_K:=\frac{1}{h(V_K)} 1_{K *V_K} *_h \tilde{1}_{V_K}
\] 
we have  $\norm{v_K}_{A(\widehat{G})} < D$ and $v_{K}|_{K}\equiv 1$.
\noindent We consider the net $\{v_K: K\subseteq \widehat{G}\ \text{compact}\}$ in $A(\widehat{G})$ where 
$v_{K_1} \preccurlyeq v_{K_2}$ whenever $v_{K_1}v_{K_2} = v_{K_1}$.
So $(v_K)_{ K\subseteq \widehat{G}}$ forms a $\norm{\cdot}_{A(\widehat{G})}$-bounded net in $A(\widehat{G}) \cap c_c(\widehat{G})$.
 Let $f\in A(\widehat{G})\cap c_c(\widehat{G})$ with $K=\supp f$. Then $v_K f= f$. Therefore, $(v_K)_{K\subseteq \widehat{G}}$ is a bounded approximate identity of $A(\widehat{G})$.

 Using ${\cal T}$ defined in the proof of Theorem~\ref{t:Fourier-of-^G}, we can define the net $(u_K)_{K\subseteq \widehat{G}}$ in $S^1(G)$ by $ u_K:={{\cal T} }^{-1}(v_K)$.
 Now, we show that $(u_K)_{K\subseteq \widehat{G}}$ satisfies all the conditions of the Theorem~\ref{t:SUM}. 
First of all, since ${\cal T}$ is an isometry from $ZL^1(G)$ onto $A(\widehat{G})$, $(u_K)_{K\subseteq \widehat{G}}$ is a $\norm{\cdot}_1$-bounded sequence in $S^1(G)$. Therefore, it forms a multiplier-bounded sequence in the Segal algebra.
Moreover, since ${\cal T}$ is an isomorphism, 
\[
u_{K_1} * u_{K_2} = {\cal T}^{-1}(v_{K_1}) * {\cal T}^{-1}(v_{K_2})={\cal T}^{-1} (v_{K_1} v_{K_2}) = {\cal T}^{-1}(v_{K_1}) = u_{K_1}
\]
for $u_{K_1}\preccurlyeq u_{K_2}$.
Toward a contradiction, suppose that $
 \sup_{K\subseteq \widehat{G}} \norm{u_K}_{S^1(G) }\leq C$
  for some $C>0$. We know that for a Segal algebra $S^1(G)$, the group $G$ is a SIN group if and only if $S^1(G)$ has
a central approximate identity which is bounded in $L^1$-norm, \cite{ko}.
So let $(e_\alpha)_{\alpha}$ be a central  approximate identity of $S^1(G)$ which is $\norm{\cdot}_1$-bounded. Since $ZS^1(G)=\overline{lin\{\chi_n\}_{n\in\Bbb{N}}}^{\norm{\cdot}_{S^1(G)}}$
and $(u_K)_{K\subseteq \widehat{G}}$ is a $\norm{\cdot}_1$-bounded approximate identity for $ZL^1(G)=\overline{lin\{\chi_n\}_{n\in\Bbb{N}}}^{\norm{\cdot}_{1}}$,   we can show that for each $\alpha$,  $\norm{e_\alpha}_{S^1(G)} =  \lim_{K \rightarrow \widehat{G}} \norm{e_\alpha * u_K}_{S^1(G)}$.
Consequently,
\[
 \norm{e_\alpha}_{S^1(G)} = \lim_{K \rightarrow \widehat{G}} \norm{e_\alpha * u_K}_{S^1(G)} \leq \lim_{K \rightarrow \widehat{G}}  \norm{e_\alpha}_1 \norm{u_K}_{S^1(G)} \leq C \norm{e_\alpha}_1.
\]  
Hence, $(e_\alpha)_\alpha$ is $\norm{\cdot}_{S^1(G)}$-bounded. But, a Segal algebra  
  cannot have a bounded approximate identity unless it coincides with the group algebra, see \cite{bu}, which contradicts the properness of $S^1(G)$. Hence, $(u_K)_{K\subseteq \widehat{G}}$ is not $\norm{\cdot}_{S^1(G)}$-bounded. So by Theorem~\ref{t:SUM}, $S^1(G)$ is not approximately amenable.

\end{proof}

\begin{cor}\label{c:Segal-of-S^1(SU(2))}
Every proper Segal algebra on $SU(2)$ is not approximately amenable.
\end{cor}

\begin{cor}\label{c:Segal-of-Product-G_i}
Let $\{G_i\}_{i\in\Ind}$ be a non-empty family of compact groups whose duals satisfy the Leptin condition, and $G=\prod_{i\in\Ind}G_i$ equipped with product topology. Then every proper Segal algebra on $G$ is not approximately amenable.
\end{cor}

\end{section}

\begin{section}{Further questions}

\begin{itemize}
\item{For which other compact groups do their duals satisfy the Leptin condition? 
 The best suggestions to study the Leptin condition for them are Lie groups which are the natural next step after $SU(2)$.}
\item{For locally compact groups, the existence of a bounded approximate identity of the Fourier algebra  implies the Leptin condition. In the hypergroup case, it seems that we cannot prove this implication. Can we find a hypergroup $H$ whose Fourier algebra has a bounded approximate identity while $H$ does not satisfies the Leptin condition.}
\end{itemize}


\end{section}

\end{section}
\vskip2.0em
\noindent {\bf Acknowledgements}\\
The author was supported by Dean's Ph.D. scholarship at the University of Saskatchewan. The author would like to express his deep gratitude to Yemon Choi and Ebrahim Samei, his supervisors, for their kind helps and constant encouragement.


 \noindent
   Mahmood Alaghmandan\\
          Department of Mathematics and Statistics,\\
          University of Saskatchewan,\\
          Saskatoon, SK S7N 5E6, CANADA\\
          E-mail: \texttt{mahmood.a@usask.ca}\\

\end{document}